\documentclass[reqno,11pt]{amsart}
\usepackage{amsthm,amssymb,latexsym,mathrsfs}
\usepackage{amsmath}
\usepackage{indentfirst}
\usepackage{color}

\usepackage[top=2.5cm,bottom=2.5cm,left=2.5cm,right=2.5cm]{geometry}
\allowdisplaybreaks[1]
\numberwithin{equation}{section}
\newtheorem{theorem}{Theorem}[section]
\newtheorem{defn}[theorem]{Definition}
\newtheorem{lemma}[theorem]{Lemma}
\newtheorem{remark}[theorem]{Remark}
\newtheorem{prop}[theorem]{Proposition}

\allowdisplaybreaks

\begin{document}

\title[Vlasov-Fokker-Planck/Navier-Stokes system]{Global weak solutions to the Vlasov-Poisson-Fokker-Planck-Navier-Stokes system}

\author{Li Chen}
\address{Lehrstuhl f\"{u}r Mathematik IV, Universit\"{a}t Mannheim,
Mannheim 68131, Germany}
\email{chen@math.uni-mannheim.de}

\author{Fucai Li}
\address{Department of Mathematics, Nanjing University,
 Nanjing 210093, P.R. China}
\email{fli@nju.edu.cn}

\author{Yue Li}
\address{Department of Mathematics, Nanjing University,
 Nanjing 210093, P.R. China}
\email{liyue2011008@163.com}

\author{Nicola Zamponi}
\address{Institute for Analysis and Scientific Computing, Vienna University of Technology, Wiedner Hauptstra\ss e 8-10, 1040 Vienna, Austria}
\email{nicola.zamponi@asc.tuwien.ac.at}

%




\begin{abstract}
We consider the compressible Vlasov-Poisson-Fokker-Planck-Navier-Stokes system  in a three dimensional bounded domain with nonhomogeneous Dirichlet boundary conditions.
The system describes the evolution of  charged  particles ensemble dispersed in an isentropic fluid.
For the adiabatic coefficient $\gamma>3/2$, we establish the global existence of weak solutions to this system
with arbitrary large initial and boundary data.
\end{abstract}

\maketitle

\section{ Introduction}

\subsection{The model}
An increasing interest in studying the kinetic-fluid models, which describe the evolution of dispersed particles in the fluid, is stimulated by extensive applications, for example, in the modeling of reaction flows of sprays, sedimentation phenomenon, wastewater treatment, sedimentation-consolidation process and rainfall formation, see \cite{BBPS,BBKT,BBE,BWC,FFS,S,SG}.

In this paper, we consider the coupling of the Vlasov-Poisson-Fokker-Planck equations and compressible Navier-Stokes system in three dimensions:
\begin{equation}\label{VNS}
\left\{
\begin{aligned}
  &\partial_tf+v\cdot\nabla_xf+{\rm{div}}_v((u-v)f)-\nabla_x\Phi\cdot\nabla_vf-\Delta_vf=0,\\
  &-\Delta_x\Phi=\int_{\mathbb{R}^3}f\,dv-c(x),\\
  &\partial_t\rho+{\rm{div}}_x(\rho u)=0,\\
  &\partial_t(\rho u)+{\rm{div}}_x(\rho u\otimes u)+\nabla_x\rho^\gamma
  -{\rm{div}}_x\mathbb{S}(\nabla u)=-\int_{\mathbb{R}^3}(u-v)f\,dv,
\end{aligned}
\right.
\end{equation}
where $f = f(t, x, v)\geq 0$ stands for the distribution of particles at time $t>0$ with respect to position $x$
and velocity $v$,
$\Phi(t,x)$ is the internal electric potential of the kinetic system,
and $\rho(t,x)$ and $u(t,x)$ represent the density and velocity of the fluid, respectively.
$\gamma>3/2$ is the adiabatic coefficient,
$c(x)$ is a given background function,
and $\mathbb{S}$ is given by
$$\mathbb{S}(\nabla u)=\mu_1(\nabla_x u+\nabla_x^\top u)+\mu_2{\rm{div}}u\,\mathbb{I},$$
where $\mu_1 $ and $\mu_2$ are coefficients of viscosity satisfying $\mu_1 > 0$ and $2\mu_1 + 3\mu_2 \geq 0$,
and $\mathbb{I}$ is the $3\times3$ identity matrix.

The kinetic equation in \eqref{VNS} describes an interacting large system with $M$ particles whose evolution is driven by the velocity $u$ of the compressible fluid system, which can be understood in the following microscopic mean field structure
\begin{align*}
dX_{M,i}(t)&=V_{M,i}(t)dt, \quad 1\leq i\leq M,\\
dV_{M,i}(t)&=\Big(-u(X_{M,i}(t))+V_{M,i}(t)-\nabla U*c(X_{M,i}(t))+\frac{1}{M}\sum_{j\neq i}^{M} \nabla U(X_{M,i}(t)-X_{M,j}(t))\Big)dt + dW_t^i,
\end{align*}
where $(X_{M,i}(t),V_{M,i}(t))_{1\leq i\leq M}$ are the positions and velocities of a group of $M$ particles, $(W^i_t)_{1\leq i\leq N}$ is a set of independent Brownian motions, and $U$ is the Newtonian potential in three dimensions.
In this particle formulation, the driven force $-u(X_{M,i}(t))+V_{M,i}(t)$ for the $i$-th particle is coupled with the fluid system through the fluid velocity $u$. The rigorous derivation of Vlasov-Poisson-Fokker-Planck through mean field limit was recently shown in \cite{CCS, HLP}. In the last decades, there have been many notable works in the derivation of one particle effective equation through the mean field limit of first (or second) order particle systems (in both deterministic or stochastic setting), for example, \cite{CCH,CDJ,CDHJ,CHJZ,G,J,JW,LP}, to name a few. However, due to the singular coupling of the particle system on the right hand side of the momentum equation, a rigorous derivation of the whole Vlasov-Poisson-Fokker-Planck-Navier-Stokes system is still missing.

We point out that a similar incompressible version of the system \eqref{VNS} was introduced in \cite{AIS14} by   Anoshchenko,  Iegorov, and  Khruslov. And the  existence of global weak solutions to it with homogeneous Dirichlet boundary conditions was obtained by mainly using the modified Galerkin method and Schauder's fixed point
theorem.

The main goal of this paper is to prove existence of weak solutions to the system \eqref{VNS} with given initial and boundary conditions. More precisely, the setting of the problem is the following. Let $\Omega\subset\mathbb{R}^3$ be a given bounded domain with smooth boundary $\partial \Omega$.

We impose the following initial conditions
\begin{equation}\label{1}
\big(f(0,x,v),\rho(0,x),u(0,x) \big)=\big(f_0(x,v),\rho_0(x),u_0(x) \big).
\end{equation}
For the fluid equations \eqref{VNS}$_3$-\eqref{VNS}$_4$ we consider the nonhomogeneous Dirichlet boundary conditions:
\begin{align}
&\rho(t,x)=\rho_B(x),\quad (t,x)\in (0,T)\times\Gamma_{\rm{in}},\\
&u(t,x)=u_B(x),\quad (t,x)\in (0,T)\times\partial\Omega,
\end{align}
and for the kinetic equation\eqref{VNS}$_1$ and the Poisson equation \eqref{VNS}$_2$,  we consider the following inflow boundary conditions:
\begin{align}
 &\gamma^- f(t,x,v)=g(t,x,v),\quad (t,x,v)\in (0,T)\times \Sigma^-,\\
  &\Phi(t,x)=0,\quad (t,x)\in (0,T)\times\partial\Omega,\label{2}
\end{align}
where  $\Gamma_{\rm{in}}:=\{x\in\partial\Omega\,\big|\,u_B\cdot \nu(x)<0\}$ and $\gamma ^{-}f(t,x,v)$ is the trace of $f$, $\Sigma^{-}:=\{(x,v)\in\partial\Omega\times\mathbb{R}^3| v\cdot \nu(x)<0\}$.
Here $\nu(x)$ denotes the outward unit normal vector to $x\in\partial\Omega$.
Furthermore, we use $\gamma^+f$ to denote the trace of $f$ on $(0,T)\times\Sigma^+$ and
$\Sigma^{+}:=\{(x,v)\in\partial\Omega\times\mathbb{R}^3| v\cdot \nu(x)>0\}$.

\subsection{Some previous results}
Kinetic-fluid models have been studied in different contexts depending on, for example, the compressibility of the fluid, the types of friction force and the magnitude of initial data.
In the following, we summarize some results in two categories: compressible kinetic-fluid model and incompressible kinetic-fluid model.
We will restrict ourselves to the kinetic-fluid model with non-density dependent viscosity.

The typical compressible Vlasov-Fokker-Planck-Navier-Stokes system reads:
\begin{equation}\label{s1}
\left\{
\begin{aligned}
  &\partial_tf+v\cdot\nabla_xf+{\rm{div}}_v((u-v)f)-\Delta_vf=0,\\
  &\partial_t\rho+{\rm{div}}_x(\rho u)=0,\\
  &\partial_t(\rho u)+{\rm{div}}_x(\rho u\otimes u)+\nabla_x\rho^\gamma- \Delta_xu=-\int_{\mathbb{R}^3}(u-v)f\,dv.
\end{aligned}
\right.
\end{equation}
Chae, Kang and Lee \cite{CKL'} showed that the system \eqref{s1} possesses a global classical solution close to an equilibrium
$
\big(\bar f, \bar \rho, \bar u\big)=\big(e^{- {|v|^2}/{2}}, 1,0\big)
$ and proved that the solution converges to the equilibrium exponentially fast in $\mathbb{T}^3$.
Mellet and Vasseur \cite{MV} established the global existence of weak solutions to the system \eqref{s1} for
large initial data, furthermore, they \cite{MV'} studied the asymptotic analysis of the solutions.
In the case of nonhomogeneous Dirichlet boundary condition for the velocity of the fluid, the global existence of weak solutions to the system \eqref{s1} was given by the second author \cite{Li}.
Moreover, this result was extended to the system \eqref{s1} with a local alignment force in \cite{LL2}.
For the case that the friction force is dependent on the density $\rho$, Li, Mu and Wang \cite{LMW} established the global well-posedness of a strong solution when the initial data is a small perturbation of some given equilibrium.

Now we focus on the incompressible kinetic-fluid system.
Chae, Kang and Lee \cite{CKL} derived the global existence of weak and classical solutions to the Vlasov-Fokker-Planck-Navier-Stokes system.
Boudin, Desvillettes, Grandmont and Moussa \cite{BDGM} proved that the Vlasov-Navier-Stokes system possesses global weak solutions in three-dimensional
periodic case.
Furthermore, a similar result in three-dimensional bounded domain was obtained in \cite{Yu}.
When the friction force depends on the density,
Wang and Yu \cite{WY} derived the global well-posedness of weak solutions to the Vlasov-Navier-Stokes equations.
Anoshchenko, Khruslov and Stephan \cite{AKS} proved that the Vlasov-Poisson-Navier-Stokes system possesses global weak solution.
Choi and Jung \cite{CJ} studied the hydrodynamic limit of an incompressible  Vlasov-Poisson-Fokker-Planck-Navier-Stokes system in a bounded domain  by the relative entropy method.

Note that when taking the distribution function $f$ equal to zero, then the system \eqref{VNS} becomes the classical compressible isentropic Navier-Stokes equations. Lions \cite{L} proved the existence of global weak solutions of the three-dimensional Navier-Stokes equations for the adiabatic constant $\gamma\geq 9/5$ with the boundary condition $u|_{\partial\Omega}=0$. Later, Feireisl, Novotn\'{y} and Petzeltov\'{a} \cite{FNP} improved the above result to $\gamma>3/2$. In \cite{JZ}, Jiang and Zhang established the global existence of axisymmetric weak solutions for $\gamma>1$ with axisymmetric initial data. Bresch and Jabin \cite{BJ} proved the global existence of appropriate weak solutions to compressible Navier-Stokes equations with general viscous stress tensor. In the case of nonhomogeneous Dirichlet boundary condition, the global existence of weak solutions was given by Plotnikov and Sokolowski \cite{PS} by means of Young measure and other tools. They first showed the Navier-Stokes equations possess weak solution when $\gamma$ is big enough, then this result was extended to the case of $\gamma>3/2$ with the help of kinetic theory. Further, Chang, Jin and Novotn\'{y} \cite{CJN} gave an other proof of this result thanks to the effective viscous flux identity, oscillations defect measure and renormalization techniques for the continuity equation in the spirt of \cite{L,FNP}.

\subsection{Our results}
Weak solutions to the problem \eqref{VNS}-\eqref{2} is defined as follows:
\begin{defn}\label{defn}
Let $T>0$ arbitrary. The function $(f,\Phi, \rho, u) : (0,T)\times\Omega \to \mathbb{R}_+\times\mathbb{R}\times\mathbb{R}_+\times\mathbb{R}^3$
is called a {\em weak solution} to the problem \eqref{VNS}-\eqref{2} on $[0,T]$ if
it possesses the regularity
\begin{gather*}
f\in L^\infty(0,T;L^1\cap L^\infty(\Omega\times\mathbb{R}^3)),\quad |v|^2f\in L^\infty(0,T;L^1(\Omega\times\mathbb{R}^3)),\\
\Phi\in L^\infty\big(0,T; W^{2,\frac{5}{3}}(\Omega)\big),\quad
\rho\in L^\infty(0,T;L^\gamma(\Omega)),\\
u\in L^2(0,T;H^1(\Omega)),\quad
\rho u\in L^\infty\big(0,T;L^\frac{2\gamma}{\gamma+1}(\Omega)\big),
\end{gather*}
and fulfills the following relations:\medskip\\
1.~{\em Weak formulation of the Vlasov-Fokker-Planck equation:} for any $\varphi\in C^\infty_c([0,T)\times\bar\Omega\times\mathbb{R}^3)$ such that $\varphi=0$ on $(0,T)\times\Sigma^+$, it holds
\begin{align*}
&\int^{T}_0\int_{\Omega\times\mathbb{R}^3}f\big(\partial_t\varphi+v\cdot\nabla_x\varphi+(u-v)\cdot\nabla_v\varphi
-\nabla_x\Phi\cdot\nabla_v\varphi
+\Delta_v\varphi\big)\,dxdvdt\\
&\qquad+\int_{\Omega\times\mathbb{R}^3}f_0\varphi(0,x,v)\,dxdv
=\int_0^T\int_{\Sigma^-}(v\cdot \nu(x))g\varphi\,d\sigma(x)dvdt.
\end{align*}
2.~{\em Weak formulation of the Poisson equation:} for any $\Psi\in C^\infty_c((0,T)\times\Omega)$, it holds
\begin{align*}
\int_0^T\int_\Omega\nabla_x\Phi\cdot\nabla_x\Psi\,dxdt=
\int_0^T\int_\Omega (n-c(x))\Psi\,dxdt,
\end{align*}
where $n:=\int_{\mathbb{R}^3}f\,dv$;\\
3.~{\em Weak formulation of the continuity equation:} for any $\psi\in C^\infty_c([0,T)\times(\Omega\cup\Gamma_{\rm{in}}))$, it holds
\begin{align*}
\int^{T}_0\int_{\Omega}(\rho\partial_t\psi+\rho u\cdot\nabla_x\psi)\,dxdt+\int_{\Omega}\rho_0\psi(0,x)\,dx=
\int_0^T\int_{\Gamma_{\rm{in}}}\rho_B u_B\cdot \nu(x)\varphi\,d\sigma(x)dt;
\end{align*}
4.~{\em Weak formulation of the momentum balance equation:} for any $\phi\in C^\infty_c((0,T)\times\Omega;\mathbb{R}^3)$, it holds
\begin{align*}
\int^{T}_{0}\int_{\Omega}\big(&\rho u\cdot\partial_t\phi+(\rho u\otimes u):\nabla_x\phi+\rho^\gamma{\rm{div}}_x\phi
-\mathbb{S}(\nabla_x u):\nabla_x\phi\\
&+(j-nu)\cdot\phi\big)\,dxdt
+\int_{\Omega}\rho_0u_0\cdot \phi(0,x)\,dx=0,
\end{align*}
where $j:=\int_{\mathbb{R}^3}vf\,dv$;\\
5.~{\em Energy balance inequality:}
for any $\tau\in (0,T)$,
\begin{align}\label{energy main}
\int_{\Omega}&\Big(\frac{1}{2}\rho|u-u_{\infty}|^2+\frac{1}{\gamma-1}\rho^{\gamma}
+\frac{1}{2}|\nabla\Phi|^2
+\int_{\mathbb{R}^3}\frac{|v|^2}{2}f\,dv\Big)(\tau)\,dx\nonumber\\
&+\int^{\tau}_0\int_{\Omega}\mathbb{S}(\nabla(u-u_\infty)):\nabla(u-u_\infty)\,dxdt
+\int^{\tau}_0\int_{\Sigma^-}(v\cdot \nu(x))\frac{|v|^2}{2}g\,d\sigma(x)dvdt\nonumber\\
&+\int_0^\tau\int_{\Gamma_{\rm{in}}}\rho^\gamma|u_B\cdot \nu(x)|\,d\sigma(x)dt
+\int_0^\tau\int_{\Gamma_{\rm{out}}}\frac{1}{\gamma-1}\rho^\gamma|u_B\cdot\nu(x)|\,d\sigma(x)dt\nonumber\\
\leq&\int_{\Omega}\Big(\frac{1}{2}\rho_0|u_0-u_{\infty}|^2+\frac{1}{\gamma-1}\rho_0^{\gamma}+\frac{1}{2}|\nabla\Phi_0|^2
+\int_{\mathbb{R}^3}\frac{|v|^2}{2}f_0\,dv\Big)\,dx\nonumber\\
&+\int^{\tau}_0\int_{\Omega}\big(-\rho^{\gamma}{\rm{div}}u_{\infty}-\mathbb{S}(\nabla u_{\infty}):\nabla(u-u_\infty)
-\rho u\cdot\nabla u_\infty\cdot(u-u_\infty)\big)\,dxdt\nonumber\\
&+\int^{\tau}_0\int_{\Gamma_{\rm{in}}}\frac{\gamma}{\gamma-1}\rho^{\gamma-1}\rho_{B}|u_B\cdot \nu(x)|\,d\sigma(x)dt
+3\int_{\Omega\times\mathbb{R}^3}f\,dxdv
-\int^{\tau}_0\int_{\Omega}(j-nu)\cdot u_\infty\,dxdt,
\end{align}
where $\Phi_0$ is determined by the equation $-\Delta\Phi_0=\int_{\mathbb{R}^3}f_0\,dv-c(x)$, and
$u_\infty(x)\in W^{1,\infty}(\Omega;\mathbb{R}^3)$ is an extension of $u_B(x)$, satisfying
\begin{gather}
{\rm{div}}u_\infty\geq 0\;\;\textrm{a.e.}\;\;{\rm{in}}\;\;U^-_h\equiv\{x\in\Omega\big|{\rm{dist}}(x,\partial\Omega)<h\},\;\;h>0 \;\;{\rm{is\,\,small}}.\label{extension2}
\end{gather}
\end{defn}

\begin{remark}
The existence of the extension $u_\infty$ of $u_B$ was given by  Girinon \cite{Gir}.
\end{remark}

Now we state our main result.
\begin{theorem}\label{main}
Let $\Omega\subset\mathbb{R}^3$ be a bounded domain with smooth boundary, $c(x)\in L^p(\Omega)$ $(1\leq p<+\infty)$ be a given function and $\gamma >3/2$. Suppose that the boundary data
$u_B\in C^2(\partial\Omega;\,\mathbb{R}^3)$, $\rho_B\in C(\partial\Omega)$, and $\mathop{{\rm{min}}}\limits_{\Gamma_{\rm{in}}}\rho_B=\underline{\rho}_B>0$.
Assume that the initial and boundary data are of finite energy:
\begin{align}
&0\leq f_0\in L^1\cap L^\infty(\Omega\times\mathbb{R}^3),\;\;0\leq g\in L^1\cap L^\infty((0,T)\times\Sigma^-), \label{fv}\\
&\int_{\Omega}\rho_0\,dx>0,\;\;\;\;\int^T_0\int_{\Sigma^-}|v|^2g(t,x,v)|v\cdot \nu(x)|\,d\sigma(x)dvdt<\infty, \\
&\int_{\Omega}\Big(\frac{1}{2}\rho_0|u_0-u_\infty|^2+\frac{1}{\gamma-1}\rho_0^\gamma
+\int_{\mathbb{R}^3}\frac{|v|^2}{2}f_0\,dv\Big)\,dx
<\infty,\label{mv}
\end{align}
then for any $T>0$, the problem \eqref{VNS}-\eqref{2} possesses at least one global weak solution $(f,\Phi,\rho,u)$ on $[0,T]$.
\end{theorem}

Compared with the weak existence result obtained in \cite{LL2}, where the
Vlasov-Fokker-Planck-Navier-Stokes equations with alignment force term is considered, one can clearly see that the main difficulty comes from the coupling with the Poisson equation \eqref{VNS}$_2$. In fact, the weak solution theory for Vlasov-Poisson-Fokker-Planck system has been established in a classical work \cite{C}. However, the $k$-th moment estimates for the iteration scheme can only be obtained with additional assumption on the solution. If we perform the same iteration scheme for the kinetic part of our system, due to the complicated structure of this system, we can not arrive at the similar moment estimates. Therefore, instead of using the method developed in \cite{C}, we construct a new approximation scheme to decouple the Vlasov-Poisson-Fokker-Planck system and apply a special case of Leray-Schauder fixed point theorem to complete a sequence of well-defined approximate solutions. Consequently, the global weak solution is obtained by entropy estimate and compactness argument.

Next we give a brief description of the method in handling with the kinetic equation.  For a given $\Phi^*\in L^\infty(0,T;W^{1,\infty}(\Omega))$, we can construct a unique solution $f$ to the Vlasov-Fokker-Planck equation with the help of
a main result in \cite{C}. More precisely, there exists a unique weak solution to the following linear equation:
\begin{align*}
\partial_t f +v\cdot \nabla_x f+{\rm{div}}_v((\nabla\Phi^*+\tilde u\chi_{\{|\tilde u|\leq N\}}-v)f)-\Delta_vf=0,
\end{align*}
where $\chi$ is a cut-off function for the fluid velocity. For more detailed technics for the fluid part, we refer to Section \ref{app}.
The next thing to do is to find a unique solution $\Phi$ to the Poisson equation $-\Delta \Phi=\int_{\mathbb{R}^3} f dv-c(x)$.
However, a $W^{1,\infty}$ bound for $\Phi$ can not be obtained by strong solution theory. Therefore, in order to close the definition of the map, we need to introduce the regularization of the Poisson equation. Furthermore, since the moment estimates are only valid for fixed point of this map, instead of using the iteration scheme as in \cite{C}, we use a special type of Leray-Schauder fixed point theorem to establish the existence of $\nabla\Phi$.
More details can be found in Section \ref{app}.

The rest of this paper is organized as follows.
Section \ref{app} is devoted to constructing the approximate solutions to the problem \eqref{VNS}-\eqref{2}.
In Section \ref{l}, we pass to the limits on regularized parameters in turn and complete the proof of Theorem \ref{main}.

\section{Approximate solutions}\label{app}

Our first goal of this section is to solve the regularized system:
\begin{align}
&\partial_tf+v\cdot \nabla_xf+{\rm{div}}_v\big((u-v)f\big)-\nabla_x\Phi\cdot\nabla_vf-\Delta_vf=0,\label{3} \\
&-\Delta\Phi-\varepsilon\Delta^{2m+1}\Phi=\int_{\mathbb{R}^3}f\,dv-c(x),\label{4}\\
&\partial_t\rho+{\rm{div}}(\rho u)=\varepsilon\Delta\rho,\label{5} \\
&\partial_t(\rho u)+{\rm{div}}(\rho u\otimes u)+\nabla\rho^{\gamma}+\delta\nabla\rho^\beta
+\varepsilon\nabla\rho\cdot\nabla u
\nonumber\\
&\qquad\quad ={\rm{div}}\mathbb{S}(\nabla u)
+\varepsilon{\rm{div}}(|\nabla(u-u_\infty)|^2\nabla(u-u_\infty))
-\int_{\mathbb{R}^3}(u-v)f\,dv,\label{6}
\end{align}
where $m\in \mathbb{N}$ is fixed large as necessary, $\varepsilon>0$, $\delta>0$ and $\beta>{\rm{max}}\{\gamma,9/2\}$.
The system \eqref{3}-\eqref{6} is equipped with the following initial and boundary data:
\begin{align}\label{7}
\big(f(0,x,v),\rho(0,x),u(0,x)\big)=v(f_{0\varepsilon}(x,v),\rho_0(x),u_0(x)\big),
\end{align}
and
\begin{align}
&\gamma^-f(t,x,v)\big|_{(0,T)\times\Sigma^-}=g_\varepsilon(t,x,v),\label{8}\\
&\Phi|_{(0,T)\times\partial\Omega}=\Delta\Phi|_{(0,T)\times\partial\Omega}=\cdot \cdot \cdot
=\Delta^{2m}\Phi|_{(0,T)\times\partial\Omega}=0,\label{8'}\\
&(-\varepsilon\nabla\rho+\rho u)\cdot \nu(x)\big|_{(0,T)\times\partial\Omega}=
\left\{
       \begin{array}{ll}
       \rho_Bu_B\cdot \nu(x) & \;\;\;{\rm{on}}\;\;\;\Gamma_{\rm{in}},\\
       \rho u_B\cdot \nu(x)&\;\;\;{\rm{on}}\;\;\;\partial\Omega\backslash\Gamma_{\rm{in}},
       \end{array}
\right.\label{9}\\
&u(t,x)\big|_{(0,T)\times\partial \Omega}=u_B(x),\label{10}
\end{align}
where $f_{0\varepsilon}$ and $g_\varepsilon$
are approximate sequences of $f_0$ and $g$, respectively. They satisfy \eqref{fv}-\eqref{mv} uniformly with respect to $\varepsilon$ and
\begin{gather*}
\int_{\Omega\times\mathbb{R}^3}|v|^\kappa f_{0\varepsilon}(x,v)\,dxdv<\infty,\quad \forall\; \kappa\in[0,\kappa_0]\;\; {\rm{with}}\;\; \kappa_0\geq 5,\\
\int_0^T\int_{\Sigma^-}|v|^\kappa g_{\varepsilon}|v\cdot\nu(x)|\,d\sigma(x)dvdt<\infty,\quad \forall\; \kappa\in[0,\kappa_0]\;\; {\rm{with}}\;\; \kappa_0\geq 5.
\end{gather*}
Here the initial and boundary data verify the assumptions of Theorem \ref{main} and
\begin{equation}\label{11}
\left\{
\begin{aligned}
  &u_0\in L^2(\Omega),\quad \rho_0\in W^{1,2}(\Omega),\\
  &0<\underline{\rho}\leq\rho_0(x)\leq\bar{\rho}<\infty,\quad x\in\Omega,\\
  &0<\underline{\rho}\leq\rho_B(x)\leq\bar{\rho}<\infty,\quad x\in\Gamma_{\rm{in}}.
\end{aligned}
\right.
\end{equation}
We still use $(\rho_0,u_0,\rho_B)$ to denote the initial and boundary data to the system \eqref{3}-\eqref{6} and will emphasize the dependence
of the parameter if it is needed.

In the rest of this section, we focus on the solvability of the problem \eqref{3}-\eqref{10}.

\subsection{Decoupled system}\label{2.1.1}
To establish the global existence of weak solutions to the problem \eqref{3}-\eqref{10}, Galerkin method is a classical one.
We introduce a finite dimensional space $X={\rm{span}}\{\Psi_i\}_{i=1}^N$, where the smooth functions
$\Psi_i(x)\,(1\leq i\leq N)$ are orthonormal in $L^2(\Omega)$.
We are now in a position to construct approximate solutions to the following system:
\begin{align}
&\partial_tf+v\cdot \nabla_xf+{\rm{div}}_v\big((\tilde u\chi_{\{|\tilde u|\leq N\}}-v)f \big)
-\nabla_x\Phi^*\cdot \nabla_v f-\Delta_vf=0,\label{12}\\
&f|_{t=0}=f_{0\varepsilon}(x,v),\quad f|_{(0,T)\times\Sigma^-}=g_\varepsilon(t,x,v),\label{v213}\\
&-\Delta\Phi-\varepsilon\Delta^{2m+1}\Phi=\int_{\mathbb{R}^3}f\,dv-c(x),\label{13}\\
&\Phi|_{(0,T)\times\partial\Omega}=\Delta\Phi|_{(0,T)\times\partial\Omega}=\cdot\cdot\cdot
=\Delta^{2m}\Phi|_{(0,T)\times\partial\Omega}=0,\label{14}\\
&\partial_t\rho+{\rm{div}}_x(\rho u)=\varepsilon\Delta_x\rho, \label{26}\\
&\rho|_{t=0}=\rho_0(x),\quad
(-\varepsilon\nabla\rho+\rho u)\cdot \nu(x)\big|_{(0,T)\times\partial\Omega}=
\left\{
       \begin{array}{llr}
       \rho_Bu_B\cdot \nu(x)& \;\;{\rm{on}}\;\;\Gamma_{\rm{in}},\\
       \rho u_B\cdot \nu(x)&\;\;{\rm{on}}\;\;\partial\Omega\backslash\Gamma_{\rm{in}},
       \end{array}
\right.\label{27}\\
&\partial_t(\rho u)+{\rm{div}}(\rho u\otimes u)+\nabla\rho^{\gamma}+\delta\nabla\rho^{\beta}+\varepsilon\nabla\rho\cdot\nabla u\nonumber\\
&\qquad ={\rm{div}}\mathbb{S}(\nabla u)+\varepsilon{\rm{div}}(|\nabla(u-u_\infty)|^{2}\nabla(u-u_\infty))
-\chi_{\{|\tilde{u}|\leq N\}}\int_{\mathbb{R}^3}(\tilde u-v)f\,dv, \label{28}\\
&u|_{t=0}=u_0(x),\quad u(t,x)\big|_{(0,T)\times\partial\Omega}=u_B(x),\label{29}
\end{align}
where $(\tilde{u},\nabla\Phi^*)\in L^2(0,T;L^2(\Omega))\times L^\infty(0,T;W^{4m,\frac{5}{3}}(\Omega))$ is a given pair and $\chi$ is continuous.

For this purpose, we commence by recalling a known result due to \cite{MV}.
\begin{lemma}\label{1'}
If one has
\begin{align*}
&\|f\|_{L^\infty((0,T)\times\Omega\times\mathbb{R}^3)}\leq M, \\
&\int_{\Omega\times \mathbb{R}^3}|v|^\kappa f(t,x,v)\,dxdv\leq M,\;\; t\in[0,T],\;\; \kappa\in[0,\kappa_0],
\end{align*}
where $\kappa_0>0$,
then there exists a constant $C(M)$ such that
\begin{align*}
&\Big\|\int_{\mathbb{R}^3}f(t,x,v)\,dv\Big\|_{L^p(\Omega)}\leq C(M), \quad p\in\Big[1,\frac{\kappa_0+3}{3}\Big],\\
&\Big\|\int_{\mathbb{R}^3}vf(t,x,v)\,dv\Big\|_{L^p(\Omega)}\leq C(M), \quad p\in\Big[1,\frac{\kappa_0+3}{4}\Big],
\end{align*}
for any $t\in[0,T]$.
\end{lemma}
It is a direct consequence of \cite{C} that there exists a unique solution $f\in L^\infty(0,T;L^1\cap L^\infty(\Omega\times\mathbb{R}^3))$
to the problem \eqref{12}-\eqref{v213}.
Moreover, the solution $f$ satisfies
\begin{align}
&\partial_t\int_{\mathbb{R}^3}f\,dv=-{\rm{div}}_x\int_{\mathbb{R}^3}vf\,dv,\label{v24}\\
&\|f\|_{L^\infty(0,T;L^p(\Omega\times\mathbb{R}^3))}
\leq e^{\frac{3T}{p'}}(\|f_{0\varepsilon}\|_{L^p(\Omega\times\mathbb{R}^3)}
+\|g_\varepsilon\|_{L^p((0,T)\times\Sigma^-)}),\quad 1\leq p\leq \infty,\label{v25}\\
&\|\gamma^+f\|_{L^p((0,T)\times\Sigma^+)}
\leq e^{\frac{3T}{p'}}(\|f_{0\varepsilon}\|_{L^p(\Omega\times\mathbb{R}^3)}
+\|g_\varepsilon\|_{L^p((0,T)\times\Sigma^-)}),\quad 1\leq p\leq \infty,\label{v26}
\end{align}
where $p'$ verifies $\frac{1}{p}+\frac{1}{p'}=1$, and for any $1\leq l<\infty$,
\begin{align}\label{v23}
\frac{d}{dt}\int_{\Omega\times\mathbb{R}^3}|v|^lf\,dxdv
&=-l\int_{\Omega\times\mathbb{R}^3}|v|^lf\,dxdv+l(l+1)\int_{\Omega\times\mathbb{R}^3}|v|^{l-2}f\,dxdv\nonumber\\
&-\int_{\Sigma^+}|v\cdot\nu(x)||v|^l\gamma^+f\,d\sigma(x)dv+\int_{\Sigma^-}|v\cdot\nu(x)||v|^lg_\varepsilon\,d\sigma(x)dv\nonumber\\
&+l\int_{\Omega\times\mathbb{R}^3}\tilde{u}\chi_{\{|\tilde{u}|\leq N \}}\cdot v|v|^{l-2}f\,dxdv
-l\int_{\Omega\times\mathbb{R}^3}\nabla\Phi^*\cdot v|v|^{l-2}f\,dxdv.
\end{align}
This together with the imbedding $W^{4m,\frac{5}{3}}(\Omega)\hookrightarrow L^\infty(\Omega)$ (m is big enough) ensure us to derive that
\begin{align*}
\big\||v|^lf\big\|_{L^\infty(0,T;L^1(\Omega\times\mathbb{R}^3))}\leq C\Big(N,\varepsilon,T,\|\nabla\Phi^*\|_{L^\infty\big(0,T;W^{4m,\frac{5}{3}}(\Omega)\big)}\Big).
\end{align*}
With this observation at hand, and combining \eqref{v25} and Lemma \ref{1'}, we have
\begin{align}
&\Big\|\int_{\mathbb{R}^3}f\,dv \Big\|_{L^\infty\big(0,T;L^\frac{l+3}{3}(\Omega)\big)}\leq
C\Big(N,\varepsilon,T,\|\nabla\Phi^*\|_{L^\infty\big(0,T;W^{4m,\frac{5}{3}}(\Omega)\big)}\Big),\label{v214}\\
&\Big\|\int_{\mathbb{R}^3}vf\,dv \Big\|_{L^\infty\big(0,T;L^\frac{l+3}{4}(\Omega)\big)}\leq
C\Big(N,\varepsilon,T,\|\nabla\Phi^*\|_{L^\infty\big(0,T;W^{4m,\frac{5}{3}}(\Omega)\big)}\Big)\label{v215}.
\end{align}

Based on \eqref{v214} and the properties of the elliptic problem (\cite{Ab}), we infer that there exists a unique solution $\Phi$ to the following problem
\begin{equation}\label{15}
\left\{
\begin{aligned}
  &-\Delta\Phi-\varepsilon\Delta^{2m+1}\Phi=\int_{\mathbb{R}^3}f\,dv-c(x),\\
  &\Phi|_{(0,T)\times\partial\Omega}=\Delta\Phi|_{(0,T)\times\partial\Omega}=\cdot\cdot\cdot=
  \Delta^{2m}\Phi|_{(0,T)\times\partial\Omega}=0,
\end{aligned}
\right.
\end{equation}
where $f$ is the weak solution to system \eqref{12}-\eqref{v213}. Furthermore,
\begin{align*}
&\|\Phi\|_{L^\infty\big(0,T;W^{4m+2,\frac{5}{3}}(\Omega)\big)}\leq C(\varepsilon)
\Big(\Big\|\int_{\mathbb{R}^3}f\,dv \Big\|_{L^\infty\big(0,T;L^\frac{5}{3}(\Omega)\big)}+\|c(x)\|_{L^\frac{5}{3}(\Omega)} \Big),\\
&\|\Phi\|_{L^\infty\big(0,T;W^{2,\frac{5}{3}}(\Omega)\big)}\leq C
\Big(\Big\|\int_{\mathbb{R}^3}f\,dv \Big\|_{L^\infty\big(0,T;L^\frac{5}{3}(\Omega)\big)}+\|c(x)\|_{L^\frac{5}{3}(\Omega)} \Big).
\end{align*}

It is time to find an approximate solution to the Navier-Stokes system \eqref{26}-\eqref{29} when $f$
is the weak solution to the problem \eqref{12}-\eqref{v213}.
Notice that $u_N\in C([0,T];X)$ can be written as
$$u_N(t,x)=\sum^N_{i=1}\zeta_i(t)\Psi_i(x),$$
where $\zeta_i(t)$ $(i=1,2,\cdots N)$ are functions of $t$.
Thanks to \eqref{v214}-\eqref{v215}, we derive that
\begin{align*}
\Big\|\chi_{\{|\tilde u| \leq N\}}\int_{\mathbb{R}^3}(\tilde u-v)f\,dv \Big\|_{L^\infty(0,T;L^2(\Omega))}\leq
C\Big(N,\varepsilon,T,\|\nabla\Phi^*\|_{L^\infty\big(0,T;W^{4m,\frac{5}{3}}(\Omega)\big)}\Big).
\end{align*}
An argument similar to the one used in \cite{CJN} shows that the problem \eqref{26}-\eqref{29}
possesses a unique approximate solution $(\rho_{N},u_{N})$
on the whole time interval $[0,T]$.
In addition, the couple $(\rho_{N},u_{N})$ satisfies, for any $(\tau,x)\in (0,T)\times\Omega$,
\begin{equation}\label{30}
\inf_{x\in\Omega}\rho_0(x)e^{-\int_0^T\|{\rm{div}}u_N\|_{L^\infty(\Omega)}\,dt}\leq \rho_N(\tau,x)\leq
\sup_{x\in \Omega}\rho_0(x) e^{\int_0^T\|{\rm{div}}u_N\|_{L^\infty(\Omega)}\,dt},
\end{equation}
and,
\begin{align}\label{31}
&\int_{\Omega}\Big(\frac{1}{2}\rho_{N}|u_{N}-u_\infty|^2+\frac{1}{\gamma-1}\rho_{N}^{\gamma}
+\frac{\delta}{\beta-1}\rho_{N}^{\beta}+\frac{1}{2}\rho_{N}^{2}\Big)\,dx
+\varepsilon\int_0^{\tau}\int_{\Omega}|\nabla\rho_{N}|^2\,dxdt\nonumber\\
&+\frac{1}{2}\int^\tau_0\int_{\partial\Omega}\rho_{N}^2|u_B\cdot \nu(x)|\,d\sigma(x)dt
+\varepsilon\int^{\tau}_0\int_{\Omega}(\gamma\rho_{N}^{\gamma-2}+\delta\beta\rho_{N}^{\beta-2})|\nabla\rho_{N}|^2\,dxdt\nonumber\\
&+\int_0^{\tau}\int_{\Omega}\mathbb{S}(\nabla(u-u_\infty)):\nabla(u-u_\infty)\,dxdt
+\int_0^\tau\int_{\Gamma_{\rm{out}}}\Big(\frac{1}{\gamma-1}\rho_{N}^\gamma+\frac{\delta}{\beta-1}\rho_{N}^\beta \Big)|u_B\cdot \nu(x)|\,d\sigma(x)dt
\nonumber\\
&+\varepsilon\int_0^\tau\int_{\Omega}|\nabla(u_{N}-u_\infty)|^4\,dxdt
+\int_0^\tau\int_{\Gamma_{\rm{in}}}(\rho_{N}^\gamma+\delta\rho_{N}^\beta)|u_B\cdot\nu(x)|\,d\sigma(x)dt
\nonumber\\
\leq\;&\int_{\Omega}\Big(\frac{1}{2}\rho_0|u_0-u_\infty|^2+\frac{1}{\gamma-1}\rho^{\gamma}_0
+\frac{\delta}{\beta-1}\rho_0^{\beta}+\frac{1}{2}\rho_0^2\Big)\,dx
+\int^\tau_0\int_{\Gamma_{\rm{in}}}\rho_{N}\rho_B|u_B\cdot \nu(x)|\,d\sigma(x)dt \nonumber\\
&+\int^{\tau}_0\int_{\Gamma_{\rm{in}}}\Big(\frac{\gamma}{\gamma-1}\rho_{N}^{\gamma-1}
+\frac{\delta\beta}{\beta-1}\rho_{N}^{\beta-1}\Big)\rho_B|u_B\cdot \nu(x)|\,d\sigma(x)dt
-\frac{1}{2}\int^{\tau}_0\int_{\Omega}\rho_{N}^2{\rm{div}}u_{N}\,dxdt\nonumber\\
&+\int^{\tau}_0\int_{\Omega}\big(-(\rho_{N}^{\gamma}+\delta\rho_{N}^{\beta}){\rm{div}}u_\infty
-\mathbb{S}(\nabla u_\infty):\nabla(u_{N}-u_\infty)
-\rho_{N} u_{N}\cdot\nabla u_\infty\cdot(u_{N}-u_\infty)\nonumber\\
&\qquad+\varepsilon\nabla\rho_{N}\cdot\nabla(u_{N}-u_\infty)\cdot u_\infty  \big)\,dxdt
+\int^{\tau}_0\int_{\Omega\times\mathbb{R}^3}f(v-\tilde{u})\chi_{\{|\tilde u|\leq N\}}\cdot (u_{N}-u_\infty)\,dxdvdt.
\end{align}

\subsection{An argument for $(\tilde u,\nabla\Phi^*)=(u_N,\nabla\Phi_N)$}\label{fix}
Before proceeding further, we recall a special case of Leray-Schauder fixed point theorem \cite{GT}:
\begin{lemma}\label{LS}
Let $\mathscr{T}$ be a compact mapping of a Banach space $\mathscr{B}$ into itself, and suppose there exists a constant $M$ such that
\begin{align*}
\|w\|_{\mathscr{B}}<M,
\end{align*}
for all $w\in\mathscr{B}$ and $\sigma\in [0,1]$ satisfying $w=\sigma\mathscr{T}w$. Then $\mathscr{T}$ has a fixed point.
\end{lemma}
To establish the existence of approximate solution to the problem \eqref{12}-\eqref{29} with $(\tilde u,\nabla\Phi^*)$ replaced by $(u_{N},\nabla\Phi_N)$, we define an operator $\mathscr{T}$:
\begin{align*}
\mathscr{T}: L^2((0,T)\times\Omega)\times L^\infty\big(0,T;W^{4m,\frac{5}{3}}(\Omega)\big) &\rightarrow L^2((0,T)\times\Omega)\times L^\infty\big(0,T;W^{4m,\frac{5}{3}}(\Omega)\big)\\
(\tilde{u},\nabla\Phi^*) &\mapsto (u_{N},\nabla\Phi_N).
\end{align*}
Let $\{(\tilde{u_i},\nabla\Phi^*_i) \}_i$ be a uniform bounded sequence in $L^2((0,T)\times\Omega)\times L^\infty(0,T;W^{4m,\frac{5}{3}}(\Omega))$ and
$(f_{Ni},\rho_{Ni},u_{Ni})$ be the corresponding sequence solutions constructed in Subsection \ref{2.1.1}.
It is easy to see
\begin{align*}
\|\nabla\Phi_{i}^*\|_{L^\infty((0,T)\times\Omega)}\leq
C\|\nabla\Phi_i^*\|_{L^\infty\big(0,T;W^{4m,\frac{5}{3}}(\Omega)\big)}\leq C,
\end{align*}
where $C$ appeared in this subsection is independent of $i$.
With this observation at hand, and combining the inequality \eqref{v23} with $l=2$, we get
\begin{align*}
\left\||v|^2f_{Ni}\right\|_{L^\infty(0,T;L^1(\Omega\times\mathbb{R}^3))}\leq C.
\end{align*}
And in accordance with \eqref{v25} and Lemma \ref{1'}, we derive that
\begin{align}\label{v27}
\Big\|\int_{\mathbb{R}^3}f_{Ni}\,dv \Big\|_{L^\infty\big(0,T;L^\frac{5}{3}(\Omega)\big)}
+\Big\|\int_{\mathbb{R}^3}vf_{Ni}\,dv \Big\|_{L^\infty\big(0,T;L^\frac{5}{4}(\Omega)\big)}\leq C.
\end{align}
From \eqref{15}, we obtain that
\begin{align}\label{v28}
\|\Phi_{Ni}\|_{L^\infty\big(0,T;W^{4m+2,\frac{5}{3}}(\Omega)\big)}\leq C.
\end{align}
Similarly, we can deduce that
\begin{align}\label{v29}
\|\partial_t\Phi_{Ni}\|_{L^\infty\big(0,T;W^{4m+1,\frac{5}{4}}(\Omega)\big)}\leq C,
\end{align}
where we have used \eqref{v24}, \eqref{15} and \eqref{v27}.
The inequalities \eqref{v28} and \eqref{v29} allow us to use Aubin-Lions lemma to get
\begin{align}
\Phi_{Ni}\rightarrow\Phi_N \quad {\rm{in}}\quad C\big([0,T];W^{4m+1,\frac{5}{3}}(\Omega)\big).
\end{align}
For the Navier-Stokes equations, the last term on the right-hand side of the inequality \eqref{31} written with $(f_{Ni},\rho_{Ni},u_{Ni})$ can be calculated as
\begin{align*}
&\int^{\tau}_0\int_{\Omega\times\mathbb{R}^3}f_{Ni}(v-\tilde{u_i})\chi_{\{|\tilde u_i|\leq N\}}\cdot (u_{Ni}-u_\infty)\,dxdvdt\\
\leq\; &\|u_{Ni}-u_\infty\|_{L^5((0,T)\times\Omega)}\Big(\Big\|\int_{\mathbb{R}^3}vf_{Ni}\,dv \Big\|_{L^\frac{5}{4}((0,T)\times\Omega)}
+N \Big\|\int_{\mathbb{R}^3}f_{Ni}\,dv \Big\|_{L^\frac{5}{4}((0,T)\times\Omega)}\Big)\\
\leq \;& C\big(\|u_{Ni}-u_\infty\|_{L^2((0,T)\times\Omega)}+\|\nabla(u_{Ni}-u_\infty)\|_{L^2((0,T)\times\Omega)}\big )\\
\leq \;& \alpha\int_0^\tau\int_\Omega |\nabla(u_{Ni}-u_\infty)|^2\,dxdt+C,
\end{align*}
where $\alpha$ is a small constant. In addition, we notice that (\cite{CJN})
\begin{align*}
\|u_{Ni}-u_\infty\|_{H^1(\Omega)}^2\leq C\|\mathbb{S}(\nabla(u_{Ni}-u_\infty)):\nabla(u_{Ni}-u_\infty)\|_{L^1(\Omega)}.
\end{align*}
Following the same path as in \cite{CJN}, we can deal with the other terms on the right-side hand of the inequality \eqref{31} to deduce that
\begin{align}\label{v210}
\|u_{Ni}\|_{L^2(0,T;H^1(\Omega))}\leq C.
\end{align}
An argument similar to the one used in \cite{LL2} shows that there exists a constant $p\in (1,\infty)$ such that
\begin{align}\label{v211}
\|\partial_t u_{Ni}\|_{L^p(0,T;W^{-1,p}(\Omega))}\leq C,
\end{align}
where we have used \eqref{28}, \eqref{30} and \eqref{31}.
By means of \eqref{v210}, \eqref{v211} and Aubin-Lions lemma, we infer
\begin{align*}
u_{Ni}\rightarrow u_{N}\quad {\rm{in}}\quad L^2((0,T)\times\Omega).
\end{align*}

For any
$(\tilde u,\nabla\Phi^*)\in L^2((0,T)\times\Omega)\times L^\infty\big(0,T;W^{4m,\frac{5}{3}}(\Omega)\big)$ satisfying
$(\tilde u,\nabla\Phi^*)=\sigma\mathscr{T}(\tilde u,\nabla\Phi^*)=\sigma(u_N,\nabla\Phi_N)$, $\sigma\in [0,1]$,
let $(f_N,\rho_N,u_N)$ be the associated solution.
Thanks to \eqref{v24} and \eqref{15}, we can rewrite the last term on the right-hand side of \eqref{v23} with $l=2$ as
\begin{align*}
-2\int_{\Omega\times\mathbb{R}^3}\nabla\Phi^*\cdot vf_N\,dxdv&=-2\sigma\int_\Omega\nabla\Phi_N\cdot\int_{\mathbb{R}^3}vf_N\,dvdx\\
&=-2\sigma\int_\Omega\Phi_N\,\partial_t\int_{\mathbb{R}^3}f_N\,dvdx\\
&=-\sigma\frac{d}{dt}\int_\Omega(\varepsilon|\nabla^{2m+1}\Phi_N |^2+|\nabla\Phi_N|^2)\,dx.
\end{align*}
As a consequence, we have
\begin{align*}
&\frac{d}{dt}\int_{\Omega}\Big(\frac{\sigma}{2}\varepsilon|\nabla^{2m+1}\Phi_N|^2
+\frac{\sigma}{2}|\nabla\Phi_N|^2+\int_{\mathbb{R}^3}\frac{|v|^2}{2}f_N\,dv\Big) \,dx
+\int_{\Sigma^\pm}(v\cdot\nu(x))\frac{|v|^2}{2}\gamma^\pm f_N\,d\sigma(x)dv\\
\leq\; & 3\int_{\Omega\times\mathbb{R}^3}f_N\,dxdv
-\int_{\Omega\times\mathbb{R}^3}|v|^2f_N\,dxdv+\int_{\Omega\times\mathbb{R}^3}\tilde u\chi_{\{|\tilde u|\leq \lambda \}}\cdot vf_N\,dxdv.
\end{align*}
It is apparent from this inequality and Gronwall inequality that
\begin{align}\label{v212}
\left\||v|^2f_N\right\|_{L^\infty(0,T;L^1(\Omega\times\mathbb{R}^3))}\leq \tilde C,
\end{align}
where $\tilde C$ is a positive constant independent of $\sigma$.
From \eqref{v25}, \eqref{15}, \eqref{v212} and Lemma \ref{1'}, we get
\begin{align}
\|\nabla\Phi_N\|_{L^\infty\big(0,T;W^{4m+1,\frac{5}{3}}(\Omega)\big)}\leq \tilde C.
\end{align}
Therefore, it holds
\begin{align*}
\|\nabla\Phi^*\|_{L^\infty\big(0,T;W^{4m,\frac{5}{3}}(\Omega)\big)}
\leq \sigma\|\nabla\Phi_N\|_{L^\infty\big(0,T;W^{4m,\frac{5}{3}}(\Omega)\big)}\leq \tilde C.
\end{align*}
We perform the same reasoning as that in the proof of \eqref{v210} to deduce that
\begin{align*}
\|\tilde u\|_{L^2((0,T)\times\Omega)}\leq \sigma\|u_N\|_{L^2((0,T)\times\Omega)}\leq \tilde C.
\end{align*}
It is possible to use Lemma \ref{LS} to yield that there exist $(u_N,\nabla\Phi_N)$ such that
$(u_N,\nabla\Phi_N)=\mathscr{T}(u_N,\nabla\Phi_N)$.
We are led to the conclusion that $(f_N,\Phi_{N},\rho_{N},u_{N})$ satisfies\\
1. for any $\varphi\in C^\infty_c([0,T)\times\bar\Omega\times\mathbb{R}^3)$ such that $\varphi=0$ on $(0,T)\times\Sigma^+$, it holds
\begin{align}\label{32}
&\int^{T}_0\int_{\Omega\times\mathbb{R}^3}f_N\big(\partial_t\varphi+v\cdot\nabla_x\varphi
+(u_{N}\chi_{\{|u_{N}|\leq N\}}-v)\cdot\nabla_v\varphi
-\nabla_x\Phi_{N}\cdot\nabla_v\varphi
+\Delta_v\varphi\big)\,dxdvdt\nonumber\\
&=-\int_{\Omega\times\mathbb{R}^3}f_{0\varepsilon}\varphi(0,x,v)\,dxdv
+\int_0^T\int_{\Sigma^-}(v\cdot \nu(x))g_\varepsilon\varphi\,d\sigma(x)dvdt;
\end{align}
2. for any $\Psi\in C_c^\infty((0,T)\times\Omega)$, it holds
\begin{align}\label{33}
&\int_0^T\int_{\Omega}\nabla\Phi_{N}\cdot \nabla\Psi\,dxdt
+\varepsilon\int_0^T\int_\Omega\nabla^{2m+1}\Phi_{N}\cdot\nabla^{2m+1}\Psi\,dxdt\nonumber\\
&=\int_0^T\int_\Omega\big(n_{N}-c(x) \big)\Psi\,dxdt,
\end{align}
where $n_{N}:=\int_{\mathbb{R}^3}f_{N}\,dv$;\\
3. for any $\psi\in C^\infty_c([0,T)\times(\Omega\cup\Gamma_{\rm{in}}))$, it holds
\begin{align}\label{34}
&\int^{T}_0\int_{\Omega}(\rho_{N}\partial_t\psi+\rho_{N} u_{N}\cdot\nabla\psi-\varepsilon\nabla\rho_{N}\cdot \nabla\psi)\,dxdt
+\int_{\Omega}\rho_0\psi(0,x)\,dx\nonumber\\
&\qquad\qquad\qquad=\int_0^T\int_{\Gamma_{\rm{in}}}\rho_B u_B\cdot \nu(x)\varphi\,d\sigma(x)dt;
\end{align}
4. for any $\phi\in C^\infty_c((0,T)\times\Omega;\mathbb{R}^3)$, it holds
\begin{align}\label{35}
\int^{T}_{0}&\int_{\Omega}\big(\rho_{N} u_{N}\cdot\partial_t\phi+(\rho_{N} u_{N}\otimes u_{N}):\nabla\phi+\rho_{N}^\gamma{\rm{div}}\phi
+\delta\rho_{N}^\beta{\rm{div}}\phi
-\varepsilon\nabla\rho_{N}\cdot\nabla u_{N}\phi\nonumber\\
&-\varepsilon|\nabla(u_{N}-u_\infty)|^2\nabla(u_{N}-u_\infty):\nabla\phi
-\mathbb{S}(\nabla u_{N}):\nabla\phi
+(j_N-n_N u_{N})\chi_{\{|u_{N}|\leq N\}}\cdot\phi\big)\,dxdt\nonumber\\
&+\int_{\Omega}\rho_0u_0\cdot \phi(0,x)\,dx=0,
\end{align}
where $j_N:=\int_{\mathbb{R}^3}vf_N\,dv$.

Moreover, the energy inequalities \eqref{v23} and \eqref{31} with $(\tilde u,\nabla\Phi^*)$ replaced by $(u_{N},\nabla\Phi_N)$ hold.

\subsection{Uniform estimates independent of $N$}\label{uN}
In order to take the limit on the regularized paremeter $N$,
we derive the uniform estimates satisfied by $(f_N,\Phi_{N},\rho_{N},u_{N})$ in this subsection.
Summing \eqref{v23} with $l=2$ and \eqref{31} with $(\tilde u,\nabla\Phi^*)=(u_{N},\nabla\Phi_N)$ up, we have
\begin{align}\label{e2}
&\int_{\Omega}\Big(\frac{1}{2}\rho_{N}|u_{N}-u_\infty|^2+\frac{1}{\gamma-1}\rho_{N}^{\gamma}
+\frac{\delta}{\beta-1}\rho_{N}^{\beta}+\frac{1}{2}\rho_{N}^{2}
+\frac{\varepsilon}{2}|\nabla^{2m+1}\Phi|^2+\frac{1}{2}|\nabla\Phi|^2
+\int_{\mathbb{R}^3}\frac{|v|^2}{2}f_N\,dv\Big)\,dx\nonumber\\
&+\varepsilon\int_0^{\tau}\int_{\Omega}|\nabla\rho_{N}|^2\,dxdt
+\frac{1}{2}\int^\tau_0\int_{\partial\Omega}\rho_{N}^2|u_B\cdot \nu(x)|\,d\sigma(x)dt\nonumber\\
&+\int_0^{\tau}\int_{\Omega}\mathbb{S}(\nabla(u-u_\infty)):\nabla(u-u_\infty)\,dxdt
+\int_0^\tau\int_{\Gamma_{\rm{out}}}\Big(\frac{1}{\gamma-1}\rho_{N}^\gamma+\frac{\delta}{\beta-1}\rho_{N}^\beta\Big)|u_B\cdot\nu(x)|\,d\sigma(x)dt
\nonumber\\
&+\varepsilon\int^{\tau}_0\int_{\Omega}(\gamma\rho_{N}^{\gamma-2}+\delta\beta\rho_{N}^{\beta-2})|\nabla\rho_{N}|^2\,dxdt
+\int^{\tau}_0\int_{\Sigma^-}(v\cdot \nu(x))\frac{|v|^2}{2}g_\varepsilon\,d\sigma(x)dvdt\nonumber\\
&+\varepsilon\int_0^{\tau}\int_{\Omega}|\nabla(u_{N}-u_\infty)|^4\,dxdt
+\int_0^\tau\int_{\Gamma_{\rm{in}}}(\rho_{N}^\gamma+\delta\rho_{N}^\beta)|u_B\cdot\nu(x)|\,d\sigma(x)dt\nonumber\\
\leq\;&\int_{\Omega}\Big(\frac{1}{2}\rho_0|u_0-u_\infty|^2+\frac{1}{\gamma-1}\rho^{\gamma}_0
+\frac{\delta}{\beta-1}\rho_0^{\beta}+\frac{1}{2}\rho_0^2
+\frac{\varepsilon}{2}|\nabla^{2m+1}\Phi_0|^2+\frac{1}{2}|\nabla\Phi_0|^2
+\int_{\mathbb{R}^3}\frac{|v|^2}{2}f_{0\varepsilon}\,dv\Big)\,dx \nonumber\\
&+\int^{\tau}_0\int_{\Gamma_{\rm{in}}}\Big(\frac{\gamma}{\gamma-1}\rho_{N}^{\gamma-1}
+\frac{\delta\beta}{\beta-1}\rho_{N}^{\beta-1}\Big)\rho_B|u_B\cdot \nu(x)|\,d\sigma(x)dt\,d\sigma(x)dt
-\frac{1}{2}\int^{\tau}_0\int_{\Omega}\rho_{N}^2{\rm{div}}u_{N}\,dxdt\nonumber\\
&+\int^{\tau}_0\int_{\Omega}\big(-(\rho_{N}^{\gamma}+\delta\rho_{N}^{\beta}){\rm{div}}u_\infty
-\mathbb{S}(\nabla u_\infty):\nabla(u_{N}-u_\infty)
-\rho_{N} u_{N}\cdot\nabla u_\infty\cdot(u_{N}-u_\infty)\nonumber\\
&\qquad +\varepsilon\nabla\rho_{N}\cdot\nabla(u_{N}-u_\infty)\cdot u_\infty  \big)\,dxdt
-\int^{\tau}_0\int_{\Omega}(j_N-n_Nu_{N})\chi_{\{|u_{N}|\leq N\}}\cdot u_\infty\,dxdt\nonumber\\
&+3\int^{\tau}_0\int_{\Omega\times\mathbb{R}^3}f_N\,dxdvdt
+\int_0^\tau\int_{\Gamma_{\rm{in}}}\rho_N\rho_B|u_B\cdot \nu(x)|\,d\sigma(x)dt.
\end{align}
Using the same arguments as that in \cite{LL2}, we can carry out that the external force term of the energy inequality \eqref{e2} can be calculated as
\begin{align}\label{100}
&\Big|\int_0^\tau\int_\Omega(j_N-n_Nu_{N})\chi_{\{|u_{N}|\leq N\}}\cdot u_\infty\,dxdt \Big|\nonumber\\
\leq &\alpha\int_0^\tau\int_{\Omega}|\nabla(u_{N}-u_\infty)|^2\,dxdt
+\tilde C\int_0^\tau\int_{\Omega\times\mathbb{R}^3}|v|^2f_N\,dxdvdt+\tilde C,
\end{align}
where $\tilde C>0$ is independent of $N$ and $\varepsilon$.

Inserting \eqref{100} into the energy inequality \eqref{e2}, we obtain that
\begin{align}
&\|f_N\|_{L^\infty(0,T;L^1(\Omega\times\mathbb{R}^3))}+\|f_N\|_{L^\infty(0,T;L^\infty(\Omega\times\mathbb{R}^3))}
+\||v|^2f_N\|_{L^\infty(0,T;L^1(\Omega\times\mathbb{R}^3))}\leq C,\label{40}\\
&\|\rho_{N}|u_{N}-u_\infty|^2\|_{L^\infty(0,T;L^1(\Omega))}
+\|u_{N}-u_\infty\|_{L^2(0,T;H^1(\Omega))}\leq C,\label{41}\\
&\|\rho_{N}\|_{L^\infty(0,T;L^\beta(\Omega))}
+\varepsilon\|\nabla(u_{N}-u_\infty)\|^4_{L^4(0,T;L^4(\Omega))}\leq C,\label{42}\\
&\varepsilon\|\nabla\rho_{N}\|^2_{L^2(0,T;L^2(\Omega))}
+\varepsilon\|\nabla(\rho_{N}^{\frac{\beta}{2}})\|^2_{L^2(0,T;L^2(\Omega))}\leq C,\label{43}\\
&\Big\|\int_{\mathbb{R}^3}f_N\,dv\Big\|_{L^\infty(0,T;L^p(\Omega))}\leq C, \quad p\in\Big[1,\frac{5}{3}\Big],\label{113}\\
&\Big\|\int_{\mathbb{R}^3}vf_N\,dv\Big\|_{L^\infty(0,T;L^p(\Omega))}\leq C, \quad p\in\Big[1,\frac{5}{4}\Big],\label{114}
\end{align}
where $C>0$ is independent of $N$, but depends on $\varepsilon$.

\subsection{Taking the limit $N\rightarrow +\infty$}\label{lN}
Based on the uniform bounds obtained in Subsection \ref{uN}, we take the limit $N\rightarrow +\infty$ in the equalities \eqref{32}-\eqref{35}.
The procedure of passing to the limit in the fluid equations \eqref{34}-\eqref{35} can be found in \cite{CJN} and so is omitted.
Here we only list the key points of handling with the kinetic system for brevity.
\begin{prop}\label{p2}
There exists a subsequence $(f_N,\Phi_{N})$ (not relabeled) satisfying
\begin{align}
&f_N\stackrel{*}{\rightharpoonup}f \quad {\rm{in}}\quad L^\infty(0,T;L^p(\Omega\times\mathbb{R}^3)),\quad 1<p\leq\infty,\label{20}\\
&\int_{\mathbb{R}^3}f_N\,dv \stackrel{*}{\rightharpoonup}\int_{\mathbb{R}^3}f\,dv \quad {\rm{in}} \quad L^\infty(0,T;L^p(\Omega)),\quad
p\in \Big[1,\frac{5}{3} \Big],\label{21}\\
&\int_{\mathbb{R}^3}vf_N\,dv \stackrel{*}{\rightharpoonup}\int_{\mathbb{R}^3}vf\,dv \quad {\rm{in}} \quad L^\infty(0,T;L^p(\Omega)),\quad
p\in \Big[1,\frac{5}{4} \Big],\label{22}\\
&\Phi_{N}\rightarrow \Phi \quad {\rm{in}}\quad C([0,T];W^{1,q}(\Omega)),\quad q\in\Big[1,\frac{15}{4} \Big),\label{23}\\
&\nabla^{2m+1}\Phi_{N}\rightarrow \nabla^{2m+1}\Phi \quad {\rm{in}}\quad C([0,T];W^{1,q}(\Omega)),\quad q\in\Big[1,\frac{15}{4} \Big).\label{101}
\end{align}
\end{prop}
\begin{proof}
It is apparent from \eqref{40}, \eqref{113} and \eqref{114} that \eqref{20}-\eqref{22} hold. It only remains to prove \eqref{23} and \eqref{101}.

Using \eqref{15}, we have
\begin{align*}
-\Delta \partial_t\Phi_{N}-\varepsilon\Delta^{2m+1}\partial_t\Phi_{N}=-{\rm{div}}_x\int_{\mathbb{R}^3}vf_{N}\,dv.
\end{align*}
Meanwhile, by means of the inequality \eqref{114}, it holds
\begin{align}\label{116}
\|\partial_t\Phi_{N}\|_{L^\infty(0,T;W^{1,p}(\Omega))}\leq C,\quad p\in \Big(1,\frac{5}{4} \Big].
\end{align}
Proceeding as that in the proof of \eqref{116}, we obtain
\begin{align}\label{117}
\|\Phi_{N}\|_{L^\infty(0,T;W^{2,p}(\Omega))}\leq C,\quad p\in \Big(1,\frac{5}{3} \Big].
\end{align}
The estimates \eqref{116} and \eqref{117} allow us to use Aubin-Lions lemma to get \eqref{23}.

Similarly, we can obtain \eqref{101}.
\end{proof}
The convergence stated in Proposition \ref{p2} makes it possible to perform the limit $N\rightarrow +\infty$
in the equalites \eqref{32}-\eqref{33}.
Thanks to the weak lower semicontinuity of convex functions, we are able to pass to the limit $N\rightarrow+\infty$ in the energy inequality \eqref{e2}. In conclusion, we have proved the following proposition.
\begin{prop}\label{p1}
The problem \eqref{3}-\eqref{10} has a global weak solution $(f,\Phi,\rho,u)$ satisfying, for any $\tau\in(0,T)$,
\begin{align}\label{e1}
&\int_{\Omega}\Big(\frac{1}{2}\rho|u-u_\infty|^2+\frac{1}{\gamma-1}\rho^{\gamma}
+\frac{\delta}{\beta-1}\rho^{\beta}+\frac{1}{2}\rho^{2}
+\frac{1}{2}|\nabla\Phi|^2+\frac{\varepsilon}{2}|\nabla^{2m+1}\Phi|^2
+\int_{\mathbb{R}^3}\frac{|v|^2}{2}f\,dv\Big)\,dx\nonumber\\
&+\varepsilon\int_0^{\tau}\int_{\Omega}|\nabla\rho|^2\,dxdt
+\varepsilon\int^{\tau}_0\int_{\Omega}(\gamma\rho^{\gamma-2}+\delta\beta\rho^{\beta-2})|\nabla\rho|^2\,dxdt
+\frac{1}{2}\int^\tau_0\int_{\partial\Omega}\rho^2|u_B\cdot \nu(x)|\,d\sigma(x)dt  \nonumber\\
&+\int_0^{\tau}\int_{\Omega}\mathbb{S}(\nabla(u-u_\infty)):\nabla(u-u_\infty)\,dxdt
+\varepsilon\int_0^{\tau}\int_{\Omega}|\nabla(u-u_\infty)|^4\,dxdt
\nonumber\\
&+\int_0^\tau\int_{\Gamma_{\rm{in}}}(\rho^\gamma+\delta\rho^\beta)|u_B\cdot\nu(x)|\,d\sigma(x)dt
+\int_0^\tau\int_{\Gamma_{\rm{out}}}\Big(\frac{1}{\gamma-1}\rho^\gamma+\frac{\delta}{\beta-1}\rho^\beta \Big)|u_B\cdot\nu(x)|\,d\sigma(x)dt\nonumber\\
&+\int^{\tau}_0\int_{\Sigma^-}(v\cdot \nu(x))\frac{|v|^2}{2}g_\varepsilon\,d\sigma(x)dvdt
-\int^\tau_0\int_{\Gamma_{\rm{in}}}\rho\rho_B|u_B\cdot \nu(x)|\,d\sigma(x)dt\nonumber\\
\leq\;&\int_{\Omega}\Big(\frac{1}{2}\rho_0|u_0-u_\infty|^2+\frac{1}{\gamma-1}\rho^{\gamma}_0
+\frac{\delta}{\beta-1}\rho_0^{\beta}+\frac{1}{2}\rho_0^2
+\frac{1}{2}|\nabla\Phi_0|^2+\frac{\varepsilon}{2}|\nabla^{2m+1}\Phi_0|^2
+\int_{\mathbb{R}^3}\frac{|v|^2}{2}f_{0\varepsilon}\,dv\Big)\,dx \nonumber\\
&-\frac{1}{2}\int^{\tau}_0\int_{\Omega}\rho^2{\rm{div}}u\,dxdt
-\int^{\tau}_0\int_{\Omega}(j-nu)\cdot u_\infty\,dxdt
+3\int_{0}^\tau\int_{\Omega\times\mathbb{R}^3}f\,dxdvdt\nonumber\\
&+\int^{\tau}_0\int_{\Omega}\big(-(\rho^{\gamma}+\delta\rho^{\beta}){\rm{div}}u_\infty
-\mathbb{S}(\nabla u_\infty):\nabla(u-u_\infty)
-\rho u\cdot\nabla u_\infty\cdot(u-u_\infty)\nonumber\\
&\qquad +\varepsilon\nabla\rho\cdot\nabla(u-u_\infty)\cdot u_\infty  \big)\,dxdt
+\int_0^\tau\int_{\Gamma_{\rm{in}}}\Big(\frac{\gamma}{\gamma-1}\rho^{\gamma-1}
+\frac{\delta\beta}{\beta-1}\rho^{\beta-1} \Big)\rho_B|u_B\cdot\nu(x)|\,d\sigma(x)dt.
\end{align}
\end{prop}
\begin{remark}
The weak formulation of momentum equation \eqref{6} reads as
\begin{align*}
\int^{T}_{0}&\int_{\Omega}\big(\rho u\cdot\partial_t\phi+(\rho u\otimes u):\nabla\phi+\rho^\gamma{\rm{div}}\phi
+\delta\rho^\beta{\rm{div}}\phi
-\varepsilon\nabla\rho\cdot\nabla u\cdot\phi\nonumber\\
&-Z_\varepsilon:\nabla\phi
-\mathbb{S}(\nabla u):\nabla\phi
+(j-nu)\cdot\phi\big)\,dxdt
+\int_{\Omega}\rho_0u_0\cdot \phi(0,x)\,dx=0,
\end{align*}
for any $\phi\in C_c^\infty((0,T)\times\Omega;\mathbb{R}^3)$.
Here $Z_\varepsilon$ is the weak limit of the term $\varepsilon|\nabla(u_{N}-u_\infty)|^2\nabla(u_{N}-u_\infty)$ in the equality \eqref{35}, that is,
\begin{align*}
\varepsilon|\nabla(u_{N}-u_\infty)|^2\nabla(u_{N}-u_\infty)\rightharpoonup Z_\varepsilon \quad
{\rm{in}} \quad L^\frac{4}{3}((0,T)\times\Omega),
\end{align*}
and
\begin{align*}
\|Z_\varepsilon\|_{L^\frac{4}{3}((0,T)\times\Omega)}\rightarrow 0 \quad {\rm{as}}\quad \varepsilon\rightarrow 0,
\end{align*}
where we have used the inequality \eqref{42}.
\end{remark}

\section{Vanishing limits and the proof of Theorem \ref{main}}\label{l}
Our ultimate goal is to take limits on $\varepsilon$ and $\delta$ in turn to get the desired existence result.

\subsection{Taking the limit $\varepsilon\rightarrow 0$}\label{le}
We use $(f_\varepsilon,\Phi_\varepsilon,\rho_\varepsilon,u_\varepsilon)$ to denote the approximate solutions of this level.
With the help of regularized technique of the continuity equation, we can take limit in the weak formulation of fluid system
\eqref{5}-\eqref{7}, \eqref{9} and \eqref{10}. For a rigorous proof the reader is referred to \cite{CJN}.
Using the same arguments as that in Subsection \ref{lN}, we can easily
take limit $\varepsilon\rightarrow 0$ in the weak formulation of the problem
\eqref{3}, \eqref{4}, \eqref{7}-\eqref{8'}. Here we only list the difference.
Owing to the inequality \eqref{e1} written with $(f_\varepsilon,\Phi_\varepsilon,\rho_\varepsilon,u_\varepsilon)$
and Gronwall inequality, we derive that
\begin{align*}
\|\nabla\Phi_\varepsilon\|_{L^\infty(0,T;L^2(\Omega))}
+\sqrt{\varepsilon}\|\nabla^{2m+1}\Phi_\varepsilon\|_{L^\infty(0,T;L^2(\Omega))}\leq C,
\end{align*}
where $C>0$ appeared in this subsection is independent of $\varepsilon$. With this observation at hand, we infer that
\begin{align*}
\varepsilon\int_0^T\int_{\Omega}\nabla^{2m+1}\Phi_{\varepsilon}\cdot\nabla^{2m+1}\Psi\,dxdt
&\leq \varepsilon\|\nabla^{2m+1}\Phi_{\varepsilon}\|_{L^2((0,T)\times\Omega)}\|\nabla^{2m+1}\Psi\|_{L^2((0,T)\times\Omega)}\\
&\leq C\sqrt{\varepsilon}\rightarrow 0\quad {\rm{as}}\quad \varepsilon\rightarrow 0.
\end{align*}
Thus we can summarize what we have proved as the following proposition.
\begin{prop}
There exists a global weak solution $(f,\Phi,\rho,u)$ to the following system:
\begin{align}
&\partial_t f+v\cdot \nabla_x f
+{\rm{div}}_v((u-v)f)-\nabla_x\Phi\cdot\nabla_vf-\Delta_vf=0, \label{54}\\
&-\Delta\Phi=\int_{\mathbb{R}^3}f\,dv-c(x),\label{55}\\
&\partial_t\rho+{\rm{div}}_x(\rho u)=0, \label{56}\\
&\partial_t(\rho u)+{\rm{div}}(\rho u\otimes u)+\nabla\rho^{\gamma}+\delta\nabla\rho^{\beta}={\rm{div}}\mathbb{S}(\nabla u)
-\int_{\mathbb{R}^3}(u-v)f\,dv, \label{57}
\end{align}
with
\begin{align*}
&f(0,x,v)=f_{0}(x,v),\quad \rho(0,x)=\rho_0(x),\quad u(0,x)=u_0(x),\\
&\gamma^-f(t,x,v)\big|_{(0,T)\times\Sigma^-}=g(t,x,v),
\quad \Phi(t,x)\big|_{(0,T)\times\partial\Omega}=0,\\
&\rho(t,x)\big|_{(0,T)\times\Gamma_{\rm{in}}}=\rho_B(x),\quad
u(t,x)\big|_{(0,T)\times\partial\Omega}=u_B(x),
\end{align*}
where $(f_0,\rho_0,u_0)$ and $(g,\rho_B,u_B)$ satisfy \eqref{fv}-\eqref{mv} and \eqref{11}.
In addition, $(f,\Phi,\rho,u)$ verifies
\begin{align}\label{e3}
&\int_{\Omega}\Big(\frac{1}{2}\rho|u-u_\infty|^2+\frac{1}{\gamma-1}\rho^{\gamma}
+\frac{\delta}{\beta-1}\rho^{\beta}
+\frac{1}{2}|\nabla\Phi|^2
+\int_{\mathbb{R}^3}\frac{|v|^2}{2}f\,dv\Big)\,dx\nonumber\\
&+\int_0^{\tau}\int_{\Omega}\mathbb{S}(\nabla(u-u_\infty)):\nabla(u-u_\infty)\,dxdt
+\int^{\tau}_0\int_{\Sigma^-}(v\cdot \nu(x))\frac{|v|^2}{2}g\,d\sigma(x)dvdt\nonumber\\
&+\int_0^\tau\int_{\Gamma_{\rm{in}}}(\rho^\gamma+\delta\rho^\beta)|u_B\cdot\nu(x)|\,d\sigma(x)dt
+\int_0^\tau\int_{\Gamma_{\rm{out}}}\Big(\frac{1}{\gamma-1}\rho^\gamma+\frac{\delta}{\beta-1}\rho^\beta \Big)|u_B\cdot\nu(x)|\,d\sigma(x)dt\nonumber\\
\leq\;&\int_{\Omega}\Big(\frac{1}{2}\rho_0|u_0-u_\infty|^2+\frac{1}{\gamma-1}\rho^{\gamma}_0
+\frac{\delta}{\beta-1}\rho_0^{\beta}
+\frac{1}{2}|\nabla\Phi_0|^2
+\int_{\mathbb{R}^3}\frac{|v|^2}{2}f_{0}\,dv\Big)\,dx \nonumber\\
&+\int^{\tau}_0\int_{\Gamma_{\rm{in}}}\Big(\frac{\gamma}{\gamma-1}\rho^{\gamma-1}
+\frac{\delta\beta}{\beta-1}\rho^{\beta-1}\Big)\rho_B|u_B\cdot \nu(x)|\,d\sigma(x)dt
+3\int^{\tau}_0\int_{\Omega\times\mathbb{R}^3}f\,dxdvdt
\nonumber\\
&+\int^{\tau}_0\int_{\Omega}\big(-(\rho^{\gamma}+\delta\rho^{\beta}){\rm{div}}u_\infty
-\mathbb{S}(\nabla u_\infty):\nabla(u-u_\infty)
-\rho u\cdot\nabla u_\infty\cdot(u-u_\infty) \big)\,dxdt\nonumber\\
&-\int^{\tau}_0\int_{\Omega}(j-nu)\cdot u_\infty\,dxdt.
\end{align}
\end{prop}

\subsection{Passing the limit as $\delta\rightarrow 0$}
In this section, we take the limit $\delta\rightarrow 0$ and relax our hypotheses on the initial and boundary data.
Performing the same arguments as those in Subsection \ref{le}, we can take the limit on $\delta$ in the weak formulation of kinetic system \eqref{54}-\eqref{55}.
And the details of taking limit in the weak formulation of fluid system \eqref{56}-\eqref{57} can be found in \cite{CJN}. We omit it for brevity.
Hence the proof of Theorem \ref{main} is completed.
\hfill\qedsymbol

\medskip
\indent
{\bf Acknowledgements:}
F. Li  and Y. Li are supported by NSFC (Grant No. 12071212). And F. Li is also supported by a project funded by the Priority Academic Program Development of Jiangsu Higher Education Institutions. NZ acknowledges support from the Alexander von Humboldt Foundation (AvH), from the Austrian Science Foundation (FWF) grant P30000, and from the bilateral Croatian-Austrian Project of the Austrian Agency for International Cooperation in Education and Research (ÖAD) grant HR 19/2020.


\begin{thebibliography}{99}

\bibitem{Ab} N.B. Abdallah, Weak solutions of the initial-boundary value problem for the Vlasov-Poisson system, 17 (1994) 451-476.

\bibitem{AIS14} O. Anoshchenko,  S.  Iegorov and E.  Khruslov,   Global weak solutions of the Navier-Stokes/Fokker-Planck/Poisson linked equations, Zh. Mat. Fiz. Anal. Geom. 10 (2014), no. 3, 267--299.

\bibitem{AKS} O. Anoshchenko, E. Khruslov and H. Stephan, Global weak solutions to the Navier-Stokes-Vlasov-Poisson System, Zh. Mat. Fiz. Anal. Geom., 6 (2010) 143-182.

\bibitem{BBPS} C. Baranger, L. Boudin, P.-E. Jabin and S. Mancini, A modeling of biospray for the upper airways, ESAIM Proc., 14 (2005) 41-47.


\bibitem{BBKT} S. Berres, R. B\"{u}rger, K.H. Karlsen and E.M. Tory, Strongly degenerate parabolic-hyperbolic systems modeling polydisperse sedimentation with compression, SIAM J. Appl. Math., 64 (2003) 41-80.

\bibitem{BBE} S. Berres, R. B\"{u}rger and E.M. Tory, Mathematical model and numerical simulation of the liquid fluidization of polydisperse solid particle mixtures, Comput. Vis. Sci., 6 (2004) 67-74.

\bibitem{BDGM} L. Boudin, L. Desvillettes, C. Grandmont and A. Moussa, Global existence of solutions for the coupled Vlasov and Navier-Stokes equations,
    Differential Integral Equations, 22 (2009) 1247-1271.

\bibitem{BJ} D. Bresch and P.-E. Jabin, Global existence of weak solutions for compressible Navier-Stokes equations: Thermodynamically unstable pressure and anisotropic viscous stress tensor, Ann. of Math. (2), 188 (2018) 577-684.

\bibitem{BWC} R. B\"{u}rger, W.L. Wendland and F. Concha, Model equations for gravitational sedimentation-consolidation processes, ZAMN Z. Angew. Math. Mech.,
   80 (2000) 79-92.

\bibitem{C} J.A. Carrillo, Global weak solutions for the initial-boundary-value problems to the Vlasov-Poisson-Fokker-Planck system, Math. Methods Appl. Sci.,
   21 (1998) 907-938.

\bibitem{CCH} J.A. Carrillo, Y.-P. Choi and M. Hauray, The derivation of swarming models: mean-field limit and Wasserstein distances, In Collective dynamics from bacteria to crowds, pp. 1-46. Springer, Vienna, 2014.

\bibitem{CCS} J.A. Carrillo, Y.-P. Choi and S. Salem, Propagation of chaos for the Vlasov-Poisson-Fokker-Planck equation with a polynomial cut-off, Comm. Contemp. Math., 21 no. 04 (2019) 1850039.

\bibitem{CKL} M. Chae, K. Kang and J. Lee, Global existence of weak and classical solutions for the Navier-Stokes-Vlasov-Fokker-Planck equations,
     J. Differential Equations, 251 (2011) 2431-2465.

\bibitem{CKL'} M. Chae, K. Kang and J. Lee, Global classical solutions for a compressible fluid-particle interaction model, J. Hyperbolic Differ. Equ.,
     10 (2013) 537-562.

\bibitem{CJN} T. Chang, B.J. Jin and A. Novotn\'{y}, Compressible Navier-Stokes system with general inflow-outflow boundary data, SIAM J. Math. Anal.,
     51 (2019) 1238-1278.

\bibitem{CDJ} L. Chen, E.S. Daus and A. J\"ungel, Rigorous mean-field limit and cross-diffusion, Z. Angew. Math. Phys., 70 (2019) no. 4, Art. 122.

\bibitem{CDHJ} L. Chen, E.S. Daus, A. Holzinger and A. J\"ungel, Rigorous derivation of population cross-diffusion systems from moderately interacting particle systems, arXiv:2010.12389.

\bibitem{CHJZ} L. Chen, A. Holzinger, A. J\"ungel and N. Zamponi, Analysis and mean-field derivation of a porous-medium equation with fractional diffusion, arXiv:2109.08598.
%
%

\bibitem{CJ} Y.-P. Choi and J. Jung, On the dynamics of charged particles in an incompressible flow:
from kinetic-fluid to fluid-fluid models, arXiv: 2008.01964v2.

\bibitem{FFS} G. Falkovich, A. Fouxon and M.G. Stepanov, Acceleration of rain initiation by cloud turbulence, Nature, 219 (2002) 151-154.


\bibitem{FNP} E. Feireisl, A. Novotn\'{y} and H. Petzeltov\'{a}, On the existence of globally defined weak solutions to the Navier-Stokes equations,
     J. Math. Fluid Mech, 3 (2001) 358-392.

\bibitem{GT} D. Gilbarg and N.S. Trudinger, Elliptic partial differential equations of second order, Springer Berlin Heidelberg, (2001).

\bibitem{Gir} V. Girinon,
Navier-Stokes equations with nonhomogeneous boundary conditions in a bounded three-dimensional domain.
J. Math. Fluid Mech. 13 (2011), no. 3, 309--339.

\bibitem{G} F. Golse, The mean-field limit for a regularized Vlasov-Maxwell dynamics, Comm. Math. Phys., 310 (2012) 789-816.

\bibitem{HLP} H. Huang, J-G. Liu and P. Pickl, On the mean-field limit for the Vlasov-Poisson-Fokker-Planck system, J. Stat. Phys., 181 (2020) 1915-1965.

\bibitem{J} P.-E. Jabin, A review of the mean field limits for Vlasov equations, Kinet. Relat. Models, 7 (2014) 661-711.

\bibitem{JW} P.-E. Jabin and Z. Wang, Mean field limit for stochastic particle systems, Active Particles, Volume 1. (2017) 379-402.


\bibitem{JZ} S. Jiang and P. Zhang, Axisymmetric solutions of the 3D Navier-Stokes equations for compressible isentropic fluids, J. Math. Pures Appl.,
   82 (2003) 949-973.

\bibitem{LP} D. Lazarovici and P. Pickl, A mean field limit for the Vlasov-Poisson system, Arch. Ration. Mech. Anal., 225 (2007) 1201-1231.

\bibitem{LL2} F. Li and Y. Li, Global weak solutions for a kinetic-fluid model with local alignment force in a bounded domain, Commun. Pure Appl. Anal.,
   20 (2021) 3583-3604.

\bibitem{LMW} F. Li, Y. Mu and D. Wang, Strong solutions to the compressible Navier-Stokes-Vlasov-Fokker-Planck equations: global existence near the equilibrium and large time behavior, SIAM J. Math. Anal., 49 (2017) 984-1026.
%
%

\bibitem{Li} Y. Li, Global weak solutions for a Vlasov-Fokker-Planck/Navier-Stokes system with nonhomogeneous boundary data, Z. Angew. Math. Phys., 72 (2021) no. 2, Art. 51.

\bibitem{L} P.-L. Lions, Mathematical Topics in Fluid Mechanics-Volume 2: Compressible Models, Oxford Science Publications, Oxford (1998).

\bibitem{MV} A. Mellet and A. Vasseur, Global weak solutions for a Vlasov-Fokker-Planck/Navier-Stokes system of equations, Math. Models Methods Appl. Sci.,
   17 (2007) 1039-1063.

\bibitem{MV'} A. Mellet and A. Vasseur, Asymptotic anslysis for a Vlasov-Fokker-Planck/Navier-Stokes system of equations, Comm. Math. Phys.,
    281 (2008) 573-596.

\bibitem{PS} P. Plotnikov and J. Sokolowski, Compressible Navier-Stokes Equations. Springer-Verlag New York (2012).

\bibitem{S} W.K. Sartory, Three-component analysis of blood sedimentation by the method of characteristics, Math. Biosci., 33 (1977) 145-165.

\bibitem{SG} A. Spannenberg and K.P. Galvin, Continuous differential sedimentation of a binary suspension, Chem. Engrg. Aust., 21 (1996) 7-11.

\bibitem{WY} D. Wang and C. Yu, Global weak solution to the inhomogeneous Navier-Stokes-Vlasov equations, J. Differential Equations, 259 (2015) 3976-4008.

\bibitem{Yu} C. Yu, Global weak solutions to the incompressible Navier-Stokes-Vlasov equations, J. Math. Pures Appl., 100 (2013) 275-293.
















\end{thebibliography}
\end{document}